\documentclass{gtpart}

\usepackage{amsfonts}

\usepackage{amsthm}

\usepackage{amssymb}

\usepackage{amsmath}

\usepackage{enumerate}

\usepackage[all]{xy}

\usepackage{graphicx}

\newcommand{\B}{\mathbb{B}}

\newcommand{\E}{\mathbb{E}}

\newcommand{\N}{\mathbb{N}}

\newcommand{\FD}{\mathcal{FD}}

\newcommand{\Rek}{\text{Rep}_k(\Gamma)}

\newcommand{\Rep}{\text{Rep}(\Gamma)}

\newcommand{\Hom}{\mathrm{Hom}}

\newcommand{\Map}{\mathrm{Map}}

\newcommand{\cross}{\times}

\newcommand{\maps}{\longrightarrow}

\newcommand{\srm}[1]{\stackrel{#1}{\maps}}

\newcommand{\Id}{\textrm{Id}}

\newcommand{\e}{\emph}

\newcommand{\xmaps}{\xrightarrow}

\newcommand{\wt}[1]{\widetilde{#1}}

\newcommand{\sm}{\wedge}

\newcommand{\isom}{\cong}

\newcommand{\heq}{\simeq}

\newcommand{\Manoa}{M\=anoa}

\newcommand{\Hawaii}{Hawai\kern.05em`\kern.05em\relax i}

\def\co{\colon\thinspace}

\keyword{Baum-Connes conjecture}
\keyword{$K$-homology}
\keyword{Index theory}
\keyword{Deformation $K$-theory}

\theoremstyle{definition} \newtheorem{pasdual}{Definition}[section]

\theoremstyle{definition} \newtheorem{slant}[pasdual]{Definition}

\theoremstyle{remark} \newtheorem{pairex}[pasdual]{Example}

\theoremstyle{plain} \newtheorem{slem}[pasdual]{Lemma}

\theoremstyle{definition} \newtheorem{misbun}[pasdual]{Definition}

\theoremstyle{plain} \newtheorem{misbunlem}[pasdual]{Lemma}

\theoremstyle{definition} \newtheorem{ass}[pasdual]{Definition}

\theoremstyle{definition} \newtheorem{xfam}{Definition}[section]

\theoremstyle{definition} \newtheorem{bundef}[xfam]{Definition}

\theoremstyle{definition} \newtheorem{bunhom}[xfam]{Definition}

\theoremstyle{definition} \newtheorem{fddef}[xfam]{Definition}

\theoremstyle{definition} \newtheorem{ddef}[xfam]{Definition}

\theoremstyle{plain} \newtheorem{dlem}[xfam]{Lemma}

\theoremstyle{remark} \newtheorem{flatex}[xfam]{Example}

\theoremstyle{plain} \newtheorem{zfd}[xfam]{Proposition}

\theoremstyle{plain} \newtheorem{classifying-map}[xfam]{Lemma}

\theoremstyle{plain} \newtheorem{pairing-lemma}[xfam]{Lemma}

\theoremstyle{plain} \newtheorem{B-surj-fd-thm}[xfam]{Theorem}

\theoremstyle{plain} \newtheorem{sgfd}[xfam]{Corollary}

\theoremstyle{plain} \newtheorem{dpfd}[xfam]{Proposition}

\theoremstyle{plain} \newtheorem{fpfd}[xfam]{Lemma}

\theoremstyle{definition} \newtheorem{transdef}[xfam]{Definition}

\theoremstyle{plain} \newtheorem{transprops}[xfam]{Lemma}

\theoremstyle{plain} \newtheorem{fifd}[xfam]{Lemma}

\theoremstyle{plain} \newtheorem{fdex}[xfam]{Corollary}

\theoremstyle{remark} \newtheorem{fdnex}[xfam]{Questions}

\theoremstyle{plain} \newtheorem{bmprop}[xfam]{Proposition}

\theoremstyle{remark} \newtheorem{fdq}[xfam]{Questions}

\theoremstyle{plain} \newtheorem{universal}[xfam]{Proposition}

\theoremstyle{plain} \newtheorem{rholem}{Lemma}[section]

\theoremstyle{plain} \newtheorem{assfac}[rholem]{Proposition}

\theoremstyle{plain} \newtheorem{injcor}[rholem]{Corollary}

\theoremstyle{definition} 

\theoremstyle{plain} 

\theoremstyle{plain} 

\theoremstyle{plain}

\theoremstyle{plain} \newtheorem{slakas}{Proposition}[section]

\title{A finite dimensional approach to the strong Novikov conjecture}

\author{Daniel Ramras} 
\givenname{Daniel}
\surname{Ramras}
\address{Department of Mathematical Sciences, New Mexico State University, Las Cruces, NM 88003}
\email{ramras@nmsu.edu}

\author{Rufus Willett} 
\givenname{Rufus}
\surname{Willett}
\address{Department of Mathematics, University of \Hawaii\ at \Manoa, Honolulu, HI 96822}
\email{rufus@math.hawaii.edu}

\author{Guoliang Yu} 
\givenname{Guoliang}
\surname{Yu}
\address{Department of Mathematics, Texas A\&M University, College Station, TX 77843-3368 and Shanghai Center for Mathematical Sciences}
\email{guoliangyu@math.tamu.edu}

\arxivreference{1203.6168}  
\arxivpassword{zvuwn}

\begin{document}




\begin{abstract}
The aim of this paper is to describe an approach to the (strong) Novikov conjecture based on continuous families of finite dimensional representations: this is partly inspired by ideas of Lusztig related to the Atiyah-Singer families index theorem, and partly by Carlsson's deformation $K$--theory.   Using this approach, we give new proofs of the strong Novikov conjecture in several interesting cases, including crystallographic groups and surface groups.    The method presented here is relatively accessible compared with other proofs of the Novikov conjecture, and also yields some information about the $K$--theory and cohomology of representation spaces.
\end{abstract}

\maketitle

\section{Introduction}

The aim of this paper is to study the \emph{strong Novikov conjecture} \cite{Kasparov:1988dw} for a finitely presented group $\Gamma$.  If we assume that $\Gamma$ has a finite classifying space $B\Gamma$, one version of this conjecture states that the analytic assembly map
$$
\mu:K_*(B\Gamma)\to K_*(C^*(\Gamma))
$$
is rationally injective; here the left hand side is the $K$--homology of $B\Gamma$ and the right hand side is the $K$--theory of the maximal group $C^*$-algebra of $\Gamma$.  We give a definition of the analytic assembly map in Section \ref{asssec} below.  The strong Novikov conjecture implies the classical Novikov conjecture on homotopy invariance of higher signatures, as well as being closely related to several other famous conjectures.  

In this article, inspired by work of Lusztig \cite{Lusztig:1972dq} on the classical Novikov conjecture, we introduce an approach to the strong Novikov conjecture based on finite dimensional unitary representations of $\Gamma$.  This approach is more elementary than the other main lines of attack on this conjecture, such as Connes and Moscovici's approach via cyclic homology \cite{Connes:1990ht}, the Connes--Gromov--Moscovici approach via Lipschitz cohomology \cite{Connes:2005pt}, or the work of Miscenko \cite{Miscenko:1974ws}, Kasparov \cite{Kasparov:1988dw}, and Baum--Connes--Higson  \cite{Baum:1994pr} discussed below.

A naive version of our approach might proceed as follows.  A finite dimensional unitary representation
\begin{equation}\label{repn}
\rho:\Gamma\to U(n)
\end{equation}
of $\Gamma$ defines a vector bundle $E_\rho$ over $B\Gamma$ via a well-known balanced product construction.  $E_\rho$ defines an element $[E_\rho]$ of the $K$--theory group $K^*(B\Gamma)$ and thus a `detecting homomorphism' 
\begin{equation}\label{detect}
\rho_*:K_*(B\Gamma)\to \Z
\end{equation}
defined by pairing with $[E_\rho]$.    As is well-known\footnote{It is also a special case of Proposition \ref{assfac} below, to which we refer the reader for a proof.}, $\rho_*$ factors through the analytic assembly map $\mu$; hence $\mu(x)\neq 0$ for any $x\in K_*(B\Gamma)$ such that there exists $\rho:\Gamma\to U(n)$ with $\rho_*(x)\neq 0$.  Thus if one can find `enough' representations to detect all of $K_*(B\Gamma)$, one would have proved the strong Novikov conjecture.

Unfortunately, this approach will not work: the bundles $E_\rho$ are flat, so Chern-Weil theory tells us that any `detecting homomorphism' as in line \eqref{detect} above is rationally trivial on reduced $K$--homology.   One possible way to salvage the idea in the paragraph above is to use infinite dimensional representations.  This led to the \emph{Fredholm representations} of Miscenko \cite{Miscenko:1974ws}, and subsequently to Kasparov's \emph{KK-theory} \cite{Kasparov:1988dw}; both of these, and the closely related approach to the Novikov conjecture through the \emph{Baum-Connes conjecture} \cite{Baum:1994pr} have proved enormously fruitful.

In this paper we suggest a different approach.  
The central idea is not to use a single representation as in line \eqref{repn} above, but instead a continuous \emph{family} of representations
$$
\rho:X\to \text{Hom}(\Gamma,U(n))
$$ 
parametrized by a topological space $X$.  Such a family defines a bundle $E_\rho$ over $X\times B\Gamma$ and thus a detecting homomorphism
$$
\rho_*:K_*(B\Gamma)\to K^*(X)
$$
from the $K$--homology of $B\Gamma$ to the $K$--theory of $X$ via slant product with $[E_\rho]\in K^*(B\Gamma\times X)$.  This $\rho_*$ still factors through the analytic assembly map.  The central result of this paper is that for many interesting groups, there \emph{are} enough detecting homomorphisms of this type to `see' all of $K_*(B\Gamma)$.  The strong Novikov conjecture follows\footnote{Given the restrictions Chern--Weil theory  places on  bundles associated to representations, one may ask if similar restrictions exist for families.  Baird and the first author have shown that the Chern classes $c_i$ of the bundle associated to an $X$--shaped family of representations vanish (rationally) for $i$ \e{greater} than the rational cohomological dimension of $X$~\cite{Baird-Ramras}.  Thus, as the rational cohomological dimension of $\Gamma$ increases, it becomes necessary to use higher dimensional families to detect all of $K_*(B\Gamma)$.}.  
We also obtain some information about the $K$--theory and cohomology of the representation varieties $\Hom(\Gamma, U(n))$, which have received a good deal of attention recently from a number of authors (see, for instance, \cite{Adem-Gomez, Baird-commuting} for free abelian groups, and~\cite{Ramras-YM} for surface groups).\\

The main precursor for these ideas is Lusztig's thesis \cite{Lusztig:1972dq}, where the Atiyah-Singer index theorem for families \cite{Atiyah:1971bh} was used to study the Novikov conjecture for finitely generated free abelian groups\footnote{This is also related to \emph{Mukai duality}  for the $n$-torus \cite{mukai}.} \cite[Section 4]{Lusztig:1972dq}, as well as some non-abelian groups \cite[Section 5]{Lusztig:1972dq}.   Our paper can be seen as an attempt to develop and conceptualize the material in \cite[Sections 4 and 5]{Lusztig:1972dq}.  Developments in $K$-homology since Lusztig's work allow us to build a representation-theoretic framework that circumvents the families index theorem: indeed, from a modern point of view, the use of the families index theorem in \cite{Lusztig:1972dq} can be viewed as a way around the absence of a well-understood theory of $K$-homology at that time.   Our approach is perhaps then more elementary that that of \cite{Lusztig:1972dq}, in that it does not require the relatively deep results of \cite{Atiyah:1971bh}.

We were also motivated to put Lusztig's work in a more obviously representation theoretic-framework by Carlsson's \emph{deformation $K$-theory}: in some sense, deformation $K$-theory  develops related ideas in homotopy theory and algebraic $K$--theory.  Carlsson associates to a group $\Gamma$ a \e{spectrum} (in the sense of stable homotopy theory) $K^{\textrm{def}} (\Gamma)$, built from the (topological) category of finite dimensional unitary representations of $\Gamma$ 
(see \cite{Ramras:2007bs} for a description of the construction).  The homotopy groups of this spectrum can be described in terms of spherical families of representations~\cite{Baird-Ramras}, and the topological Atiyah--Segal map
$$\pi_* K^{\textrm{def}} (\Gamma) \maps K^{-*} (B\Gamma)$$
considered in~\cite{Baird-Ramras} might be viewed as a sort of dual to the analytic assembly map.  Our results make this `duality' more precise: they show in particular that rational surjectivity of this map implies rational injectivity of the analytic assembly map (however, from the perspective of the Novikov conjecture, there is no reason to restrict attention to spherical families, and indeed we gain some ground by allowing our families to be parametrized by arbitrary spaces $X$).\\

Our results on the Novikov conjecture are not new: the strong Novikov conjecture is known for a huge class of groups, and we are not able to add any new cases   (in fact, there is no group which is known to lie outside the scope of current results on the conjecture).  However, we think the methods of this paper are interesting, and hope they encourage connections between analytic approaches to the Novikov conjecture and some other parts of topology.  We have aimed to keep the paper as self-contained as possible, avoiding the use of complicated general theories wherever we can: we hope this makes the paper a good introduction to aspects of the theory, both for $C^*$-algebraists who know some topology, and topologists who know some $C^*$-algebra theory.

\subsection*{Outline of the paper}

Section \ref{asssec} gives an elementary approach (based on Paschke duality) to one of the slant products between $K$-theory and $K$-homology.  We then use this and the Miscenko bundle to define a version of the analytic assembly map.  

Section \ref{fdsec} introduces a class of groups that we call $\FD$ (for `flatly detectable'): roughly this consists of those groups $\Gamma$ for which all classes in $K_*(B\Gamma)$ can be detected by families of representations as described in the introduction.  We prove that $\FD$ contains $\Z$ and is closed under taking free products, direct products and finite index supergroups, whence it contains all crystallographic groups.  We also prove that $\FD$ contains surface groups: this uses results on Yang-Mills theory from \cite{Ramras-stable-moduli}.  At the end of the section we give some concrete examples of groups that are not in the class $\FD$, ask some open questions, and give an application to computing Betti numbers of representation spaces.

Finally, Section \ref{togsec} uses our explicit slant product to prove that the detecting maps arising from families of flat bundles factor through the analytic assembly map.  This completes the paper by connecting Section \ref{fdsec} to the strong Novikov conjecture.

\section{Slant products and assembly in analytic $K$--theory}\label{asssec}

In this section we use Paschke duality (as refined by Higson \cite{Higson:1995kx} and Higson--Roe \cite[Chapter 5]{Higson:2000bs}) to give a concrete description of one of the slant products in operator $K$--theory.  This slant product was perhaps first given an analytic definition by Atiyah and Singer via their families index theorem \cite{Atiyah:1971bh}, and subsequently by Kasparov in the much broader context of his bivariant $KK$-theory (see for example \cite{Kasparov:1988dw}); our approach is perhaps simpler and more direct than either of these, however.  It is inspired by (but not the same as) the slant product briefly discussed in \cite[Exercise 9.8.9]{Higson:2000bs}.   

We then use this slant product and the so-called \emph{Miscenko bundle} to give a relatively straightforward approach to the analytic assembly map
$$
\mu\co K_*(B\Gamma)\to K_*(C^*(\Gamma))
$$
in the case that $\Gamma$ is a discrete group admitting a finite classifying space $B\Gamma$.
All of this could of course be done using Kasparov's bivariant $KK$-theory \cite{Kasparov:1988dw}, but our approach seems simpler and more direct.   

See \cite[Chapters 4 and 5]{Higson:2000bs} for background information on analytic $K$--theory and the Paschke duality approach to $K$--homology theory used in what follows.

\begin{pasdual}[\cite{Higson:2000bs}, Chapter 5]\label{pasdual}
Let $A$ be a $C^*$-algebra.  A representation of $A$ on a Hilbert space $\mathcal{H}$ is said to be \emph{nondegenerate} if $\{a\xi~|~a\in A, ~\xi\in\mathcal{H}\}$ is dense in $\mathcal{H}$ (for example, if the representation is unital).  

A representation of $A$ on $\mathcal{H}$ is said to be \emph{ample} if it is nondegenerate and no non-zero element in $A$ acts as a compact operator on $\mathcal{H}$.  

Let now $\tilde{A}$ be the unitization of $A$ (even if $A$ is already unital, in which case $\tilde{A}\cong A\oplus \C$) and fix an ample representation of $\tilde{A}$.  The \emph{dual} of $A$ is the $C^*$-algebra
$$
\mathcal{D}(A):=\{T\in\mathcal{B}(\mathcal{H})~|~[T,a]\in\mathcal{K}(\mathcal{H}) \text{ for all } a\in\tilde{A}\},
$$
i.e.\ the set of operators on $\mathcal{H}$ that commute with $\tilde{A}$ up to compact operators.  It does not depend on the choice of ample representation up to non-canonical isomorphism.  Moreover, the $K$--theory groups of $\mathcal{D}(A)$, $K_*(\mathcal{D}(A))$, do not depend on the choice of ample representation up to canonical isomorphism.  For the purpose of this piece, we follow \cite[Definition 5.2.7]{Higson:2000bs}, and define the $i^\text{th}$ $K$--homology group of $A$ to be 
$$
K^i(A):=K_{1-i}(\mathcal{D}(A)).
$$
\end{pasdual}

\begin{slant}\label{slant}
Let $A$ and $B$ be $C^*$-algebras, and let $\mathcal{H}_A$ and $\mathcal{H}_B$ be ample representations of $\tilde{A}$ and $\tilde{B}$ respectively.  Let the spatial tensor product $A\otimes B$ be represented on $\mathcal{H}_A\otimes \mathcal{H}_B$ in the natural way.  Consider also the $C^*$-algebras 
$$A\otimes\mathcal{K}(\mathcal{H}_B)\subseteq \mathcal{B}(\mathcal{H}_A\otimes\mathcal{H}_B)$$
and 
$$A\otimes\mathcal{B}(\mathcal{H}_B)\subseteq \mathcal{B}(\mathcal{H}_A\otimes\mathcal{H}_B).$$
Define a function
$$\sigma=\sigma^{A,B}\co(A\otimes B)\otimes \mathcal{D}(B) \to \frac{A\otimes\mathcal{B}(\mathcal{H}_B)}{A\otimes\mathcal{K}(\mathcal{H}_B)}$$
(where we consider $\mathcal{D}(B)$ as defined using $\mathcal{H}_B$) by the formula 
\begin{equation}\label{sigmadef}
\sigma:(a\otimes b)\otimes T\mapsto a\otimes bT;
\end{equation}
note moreover that if $a\in A$, $b\in B$ and $T\in\mathcal{D}(B)$, then the elements $a\otimes b$ and $1\otimes T$ in $\mathcal{B}(\mathcal{H}_A\otimes\mathcal{H}_B)$ commute up to elements of $A\otimes\mathcal{K}(\mathcal{H}_B)$, whence it follows that $\sigma$ is actually a $*$-homomorphism.\footnote{It would perhaps be more natural to use the stable multiplier algebra $M(A\otimes\mathcal{K}(\mathcal{H}_B))$ where we have used $A\otimes\mathcal{B}(\mathcal{H}_B)$; the latter is certainly good enough for our purposes, however, and seems to have functoriality properties that are somewhat simpler to analyze.  The fact `$K_*(A\otimes\mathcal{B}(\mathcal{H}_B))=0$', which we will use shortly, is also significantly easier to prove than its analog `$K_*(M(A\otimes\mathcal{K}(\mathcal{H}_B)))=0$'. }

The $*$-homomorphism $\sigma$ thus induces a map on $K$--theory that fits into the composition 
\begin{equation}\label{preslant}
K_i(A\otimes B)\otimes K_j(\mathcal{D}(B))\to K_{i+j}(A\otimes B\otimes \mathcal{D}(B))\stackrel{\sigma_*}{\to}K_{i+j}\Big(\frac{A\otimes\mathcal{B}(\mathcal{H}_B)}{A\otimes\mathcal{K}(\mathcal{H}_B)}\Big),
\end{equation}
where the first map is the usual (external) product in operator $K$--theory (\cite[Section 4.7]{Higson:2000bs}).  Now, the definition of $K$--homology in terms of dual algebras yields $K_i(\mathcal{D}(B))=K^{1-i}(B)$.  Moreover, using that $K_*(A\otimes\mathcal{B}(\mathcal{H}_B))=0$ (noting that ampleness implies that $\mathcal{H}_B$ is infinite dimensional, this follows from an easy Eilenberg swindle argument, just as for $\mathcal{B}(\mathcal{H}_B)$ itself) and the long exact sequence in $K$--theory we have natural isomorphisms
$$
K_i\Big(\frac{A\otimes\mathcal{B}(\mathcal{H}_B)}{A\otimes\mathcal{K}(\mathcal{H}_B)}\Big)\cong K_{i-1}(A\otimes\mathcal{K}(\mathcal{H}_B))\cong K_{i-1}(A).
$$
Thus line \eqref{preslant} is equivalent to a map
$$
K_i(A\otimes B)\otimes K^j(B)\to K_{(i+(1-j))-1}(A)=K_{i-j}(A).
$$
We call the map in the line above the \emph{slant product} in operator $K$--theory.  If $x\in K_i(A\otimes B)$ and $y\in K^j(B)$, we denote their slant product by $x/y\in K_{i-j}(A)$.
\end{slant}

Appendix \ref{kasap} gives a proof that this slant product is the same as the relevant special case of the Kasparov product (and thus that it agrees with the standard definitions in the literature).

\begin{pairex}\label{pairex}
Say in the above that $A=\C$, so that $A\otimes B$ is canonically isomorphic to $B$, and the slant product reduces to a pairing
$$
K_i(B)\otimes K^j(B)\to K_{i-j}(\C)\cong\left\{\begin{array}{ll} \Z & i=j \text{ mod } 2 \\ 0 & \text{otherwise} \end{array}\right..
$$
This pairing can be identified with the usual pairing between $K$--theory and $K$--homology as we now explain.  Assume throughout for simplicity that $B$ is unital (which is in any case all we will need).

Assume also, at least for the moment, that $i=j=0$.  It will suffice to show that the pairing above agrees with the usual pairing between $K$--homology and $K$--theory when $[p]\in K_0(B)$ is a class represented by some projection $p\in M_n(B)$, and $[u]\in K^0(B)$ is represented by some unitary $u\in \mathcal{D}(B)$.

Now, according to \cite[Proposition 4.8.3]{Higson:2000bs}, the image of the element 
$$
[p]\otimes [u]\in K_0(B)\otimes K_0(\mathcal{D}(B))
$$
under the product map to $K_0(B\otimes \mathcal{D}(B))$ can be represented by the unitary
$$
p\otimes u+(1-p)\otimes 1\in M_n(B)\otimes \mathcal{D}(B)\cong M_n(B\otimes \mathcal{D}(B)).
$$
Let $\mathcal{Q}(\mathcal{H}_B)=\mathcal{B}(\mathcal{H}_B)/\mathcal{K}(\mathcal{H}_B)$ denote the Calkin algebra and for $x\in\mathcal{B}(\mathcal{H}_B)$ write $\overline{x}$ for its image under the quotient map $\mathcal{B}(\mathcal{H}_B)\to\mathcal{Q}(\mathcal{H}_B)$.  Write $u_n$ for the element of $M_n(\mathcal{B}(\mathcal{H}_B))$ with all diagonal entries $u$, and all other entries zero.  Then it is not hard to check that the natural extension of $\sigma^{\C,B}$ to the matrix algebra $M_n(B\otimes \mathcal{D}(B))$ acts as follows:
$$
\sigma^{\C,B}\co p\otimes u+(1-p)\otimes 1\mapsto \overline{pu_n+(1-p)}\in\frac{M_n(\mathcal{B}(\mathcal{H}_B))}{M_n(\mathcal{K}(\mathcal{H}_B))}\cong M_n(\mathcal{Q}(\mathcal{H}_B)).
$$
Using that $p$ is a projection, and that $p$ and $u_n$ commute up to $M_n(\mathcal{K}(\mathcal{H}_B))$, we have that
$$
\overline{pu_n+(1-p)}=\overline{p^2u_n+(1-p)}=\overline{pu_np+(1-p)},
$$
whence the slant product of $[p]$ and $[u]$ is equal to the image of the class of $\overline{pu_np+1-p}$ in 
$K_1(\mathcal{Q}(\mathcal{H}_B))$ under the boundary map
$$
\partial ~: ~K_1(\mathcal{Q}(\mathcal{H}_B))\to K_0(\mathcal{K}(\mathcal{H}_B))\cong\Z.
$$
In this special case, however, this boundary map is concretely realized by the formula
$$
K_1(\mathcal{Q(\mathcal{H}_B)})\to\Z,~~[\overline{v}]\mapsto \text{Index}(v)
$$
(see for example \cite[Proposition 4.8.8]{Higson:2000bs}); note that if $\overline{v}\in M_n(\mathcal{Q}(\mathcal{H}_B))$ is unitary, then $v\in M_n(\mathcal{B}(\mathcal{H}_B))$ is Fredholm by Atkinson's theorem, so this makes sense.  Our conclusion, finally, is that the slant product is given by the integer
$$
\text{Index}(pu_np+1-p)
$$ 
which is the formula for the pairing between $K$--homology and $K$--theory from \cite[Section 7.2]{Higson:2000bs}.  The only other case of interest is $i=j=1$; this works analogously, however, using the same product formula.
\end{pairex} 

The following lemma, giving two simple naturality properties of the slant product, will be needed later.

\begin{slem}\label{slem}
\begin{enumerate}
\item The slant product
$$
K_i(A\otimes B)\otimes K^j(B)\to K_{i-j}(A)
$$
is functorial in the sense that if $\phi\co A\to C$ is a $*$-homomorphism, $x\in K_*(A\otimes B)$ and $y\in K^*(B)$, then 
$$
\phi_*(x/y)=((\phi\otimes 1_B)_*x)/y
$$
as elements of $K_*(C)$.
\item Let $A$, $B$, $C$ be unital $C^*$-algebras, $x$ be a class in $K_i(A\otimes B)$, $y$ be a class in $K^j(B)$ and $z$ be a class in $K_k(C)$.  Then 
$$
(z\otimes x)/y=z\otimes (x/y)
$$
as elements of $K_{i+k-j}(C\otimes A)$.
\end{enumerate}
\end{slem}

\begin{proof} 
Look first at part (1).  With notation as in Definition \ref{slant}, note first that if 
$$
\phi\otimes 1_{\mathcal{B}(\mathcal{H}_B)}:\frac{A\otimes \mathcal{B}(\mathcal{H}_B)}{A\otimes\mathcal{K}(\mathcal{H}_B)}\to \frac{C\otimes \mathcal{B}(\mathcal{H}_B)}{C\otimes\mathcal{K}(\mathcal{H}_B)}
$$  
is the natural $*$-homomorphism induced by $\phi$, then the definition of 
$$
\sigma^{A,B}\co A\otimes B\otimes\mathcal{D}(B)\to \frac{A\otimes \mathcal{B}(\mathcal{H}_B)}{A\otimes\mathcal{K}(\mathcal{H}_B)}
$$
implies that 
\begin{equation}\label{sigrho}
\sigma^{C,B}\circ(\phi\otimes 1_B\otimes 1_{\mathcal{D}(B)})=(\phi\otimes 1_{\mathcal{B}(\mathcal{H}_B)})\circ \sigma^{A,B}.
\end{equation}
It follows then from the definition of the slant product that, up to the isomorphism 
\begin{equation}\label{boundary}
K_i\Big(\frac{A\otimes\mathcal{B}(\mathcal{H}_B)}{A\otimes \mathcal{K}(\mathcal{H}_B)}\Big)\cong K_{i-1}(A)
\end{equation}
(and similarly with $A$ replaced by $C$), the $K$--theory element $\phi_*(x/y)$ is equal to 
\begin{align*}
(\phi\otimes 1_\mathcal{B})_*(\sigma^{A,B}_*(x\otimes y))&=\sigma^{C,B}_*((\phi\otimes 1_B\otimes 1_{\mathcal{D}(B)})_*(x\otimes y)) \\
& =\sigma^{C,B}_*(((\phi\otimes 1_B)_*x)\otimes y) \\
& =((\phi\otimes 1_B)_*x)/y),
\end{align*}
where we have used line \eqref{sigrho} in the first equality and naturality of the $K$--theory product in the second.  Up to the isomorphism in line \eqref{boundary} again, this is exactly the statement of the lemma.

Part (2) is a consequence of the formula in line \eqref{sigmadef} above, and naturality properties of the $K$--theory product with respect to $*$-homomorphisms and boundary maps \cite[Proposition 4.7.6]{Higson:2000bs}. 
\end{proof}

The following $K$--theory class is important for the definition of assembly.

\begin{misbun}\label{misbun}
Let $\Gamma$ be a (finitely presented) discrete group with finite classifying space $B\Gamma$, and let $C^*(\Gamma)$ denote the maximal group $C^*$-algebra for $\Gamma$.  Let $E\Gamma$ be the universal covering space of $B\Gamma$.  Then the \emph{Miscenko bundle} for $\Gamma$, denoted $M_\Gamma$, is the bundle over $B\Gamma$ with fibres $C^*(\Gamma)$ defined as the quotient of the space $E\Gamma\times C^*(\Gamma)$ by the diagonal action
$$
g\cdot(z,a):=(gz,u_ga),
$$
 where $u_g \in C^*(G)$ is the unitary element of this $C^*$-algebra corresponding to $g$.
 \end{misbun}

\begin{misbunlem}\label{misbunlem}
The $C^*$-algebra 
$$C(B\Gamma,C^*(\Gamma))\cong C^*(\Gamma)\otimes C(B\Gamma)$$
acts naturally on the right of the space of sections of the Miscenko bundle, and this space of sections is a finitely generated projective module over $C^*(\Gamma)\otimes C(B\Gamma)$.

In particular, the Miscenko bundle defines a class
$$[M_\Gamma]\in K_*(C^*(\Gamma)\otimes C(B\Gamma)).$$
\end{misbunlem}

\begin{proof}
We first define the $C^*(\Gamma)\otimes C(B\Gamma)$ module structure on the sections of $M_\Gamma$.  Let $C_b(E\Gamma,C^*(\Gamma))$ denote the $C^*$-algebra of continuous bounded functions from $E\Gamma$ to $C^*(\Gamma)$, which admits a natural left--$\Gamma$ action defined for $z\in E\Gamma$, $g\in\Gamma$ and $f\in C_b(E\Gamma,C^*(\Gamma))$ by
$$
(g\cdot f)(z):=u_g f(g^{-1}z);
$$
the space of sections of $M_\Gamma$ then clearly identifies with the fixed point subalgebra $C_b(E\Gamma,C^*(\Gamma))^\Gamma$,  consisting of $\Gamma$--equivariant maps.  Moreover, if $\pi\co E\Gamma\to B\Gamma$ is the canonical quotient, then the formula
$$
(f \cdot h)(z):=f(z)h(\pi(z))
$$
for $f\in C_b(E\Gamma,C^*(\Gamma))^\Gamma$ and $h\in C(B\Gamma,C^*(\Gamma))$ defines a right action of the algebra $C(B\Gamma,C^*(\Gamma))$ on the sections of the Miscenko bundle; we must show that this makes this space of sections into a finitely generated projective module over $C(B\Gamma,C^*(\Gamma))$.

Note then that the Miscenko bundle is locally trivial (as it is locally isomorphic to the bundle $E\Gamma\times C^*(\Gamma)$), so there exists a finite open cover of $B\Gamma$, say $\{U_1,...,U_n\}$, such that the closure of each $U_i$ is contained in some open set $V_i$ over which the Miscenko bundle is trivial.  Let $\{\phi_i\}_{i=1}^n$ be a partition of unity subordinate to $\{U_i\}_{i=1}^n$, and for each $i$ let $\psi_i$ be a function on $B\Gamma$ that is equal to $1$ on $U_i$ and vanishes outside $V_i$.  For each $i$, let $\widetilde{U_i}$ and $\widetilde{V_i}$ be arbitrary choices of homeomorphic lifts of $U_i$, $V_i$ respectively, and by abuse of notation identify functions supported in $\widetilde{U_i}$ and $U_i$, and functions supported on $\widetilde{V_i}$ and $V_i$, without further comment.  Then the $C(B\Gamma,C^*(\Gamma))$-module map
\begin{align*}
\Phi\co  C_b(E\Gamma,C^*(\Gamma))^\Gamma & \to C(B\Gamma,C^*(\Gamma))^{\oplus n} \\
f & \mapsto \oplus_{i=1}^n (\phi_if|_{\widetilde{U_i}})
\end{align*}
includes $C_b(E\Gamma,C^*(\Gamma))^\Gamma$ as a submodule of the free module $C(B\Gamma,C^*(\Gamma))^{\oplus n}$, and is moreover split by the $C(B\Gamma,C^*(\Gamma))$-module map
\begin{align*}
\Psi\co  C(B\Gamma,C^*(\Gamma))^{\oplus n} & \to C_b(E\Gamma,C^*(\Gamma))^\Gamma \\
(f_i)_{i=1}^n & \mapsto \sum_{i=1}^n \psi_i f_i|_{V_i};
\end{align*}
this shows that $C_b(E\Gamma,C^*(\Gamma))^\Gamma$ is finitely generated and projective as required.
\end{proof}

\begin{ass}\label{ass}
Let $\Gamma$, $B\Gamma$ be as in the previous definition.  Then the \emph{analytic assembly map} is the homomorphism
$$
\mu\co K_*(B\Gamma)\to K_*(C^*(\Gamma))
$$
defined by taking the slant product with the class of the Miscenko bundle $[M_\Gamma]$, i.e.\ 
$$
\mu(x)=[M_\Gamma]/x
$$
for all $x\in K_*(B\Gamma)$.
\end{ass}

\section{Families of representations and the class $\FD$}\label{fdsec}

Throughout this section, we let $\Gamma$ denote a (finitely presented) group with finite classifying space $B\Gamma$.  For $k\in \N$, we let $U(k)$ denote the $k$--dimensional unitary group, and 
$$\Rek:=\text{Hom}(\Gamma,U(k))$$
the space\footnote{$\Rek$ is given the subspace topology inherited from the product topology on $\text{Map}(\Gamma, U(k)) = U(k)^\Gamma$.} of $k$ dimensional unitary representations of $\Gamma$.   Define
$$
\Rep:=\bigsqcup_{k=1}^\infty \Rek.
$$

\begin{xfam}\label{xfam}
Let $X$ be a finite $CW$-complex, and let 
$$
\rho\co X\to\Rep
$$
be a continuous map.  We call $\rho$ an \emph{$X$-family of representations}, or simply a \emph{family of representations}.  We write $\rho_x$, a homomorphism from $\Gamma$ to some $U(k)$, for the image of $x\in X$ under $\rho$.
\end{xfam}

Note that if $\rho:X\to \Rep$ is a family of representations, then the restriction of $\rho$ to any connected component of $X$ must take values in $\Rek$ for some fixed $k$. 

Using an $X$-family as in the above definition, one may form a vector bundle over the space $B\Gamma\times X$ in the following way.

\begin{bundef}\label{bundef}
Let $\rho\co X\to \Rep$ be a family of representations.  Write $X=X_1\sqcup\cdots \sqcup X_n$ for the decomposition of $X$ into connected components, and for each $i=1,...,n$ say the image of $\rho$ restricted to $X_i$ is contained in $\text{Rep}_{k_i}(\Gamma)$.

Let $E\Gamma$ be the universal covering space of $B\Gamma$.  Consider the space
$$\bigsqcup_{i=1}^n~E\Gamma\times X_i\times \C^{k_i}$$ equipped with the $\Gamma$ action defined by
$$
g\cdot(z,x,v):=(gz, x, \rho_x(g)v).
$$
The corresponding quotient space is a vector bundle over $B\Gamma\times X$, which we denote by $E_\rho$.

We denote by $[E_\rho]\in K^0( B\Gamma\times X)=K_0(C(X)\otimes C(B\Gamma))$ the topological $K$--theory class of this bundle.  Abusing notation, we also write $[E_\rho]$ for the element $[E_\rho]\otimes 1_\Q\in K^0(B\Gamma\times X)\otimes \Q$.
\end{bundef}

Associated to each family of representations, we now obtain a ``detecting map" as follows.

\begin{bunhom}\label{bunhom}
Let $\rho\co X\to \Rep$ and $[E_\rho]\in K^0(X\times B\Gamma)$ be as in Definitions \ref{xfam} and \ref{bundef} above.  Then taking the slant product with $[E_\rho]\in K^0(X\times B\Gamma)$ defines a homomorphism
\begin{align*}
\rho_*\co K_*(B\Gamma) & \to K^*(X)\\
x & \mapsto [E_\rho]/x.
\end{align*}
Abusing notation, we also write $\rho_*$ for the homomorphism
$$
\rho_*\otimes \text{Id}_\Q\co K_*(B\Gamma)\otimes \Q\to K^*(X)\otimes \Q
$$
induced by $\rho_*$.
\end{bunhom}

\begin{fddef}\label{fddef}
Let $\Gamma$ be a (finitely presented) group with a finite model for the classifying space $B\Gamma$.  A class $x$ in the rational $K$--homology group $K_i(B\Gamma)\otimes \Q$ is said to be \emph{flatly detectable} if there exists a family of representations 
$$
\rho\co X\to \Rep
$$
such that
$$
\rho_*(x) \in K^{-i}(X)\otimes \Q
$$
is non-zero.

A group $\Gamma$ is said to be in the class $\FD$ if it has a finite model for $B\Gamma$ and if all classes in $K_*(B\Gamma)\otimes \Q$ are flatly detectable.
\end{fddef}

The above terminology stems from the fact that when the parameter space $X$ is a point and $B\Gamma$ is a smooth manifold, the bundle $E_\rho\to B\Gamma$ has a canonical flat connection.

\begin{flatex}\label{flatex}
The trivial class $[1]\in K_0(B\Gamma)$ is always flatly detectable: one checks directly that it is detected by the trivial representation 
$$\rho:\text{pt}\to \text{Hom}(\Gamma,U(1)).$$
\end{flatex}

The rest of this section is devoted to finding examples of groups in the class $\FD$.   The first two results show that all finitely generated free groups are in $\FD$.

\begin{zfd}\label{zfd}
The group $\Z$ is in the class $\FD$.
\end{zfd}

The proof we give below is based on an (unpublished) exposition of Higson-Roe of the proof of the Novikov conjecture for $\Z^n$ in Lusztig's thesis \cite{Lusztig:1972dq}.   Corollary~\ref{sgfd} also covers this case, but for the sake of variety, we give a different proof here.

\begin{proof}
We may of course take $B\Z$ to be a copy of the circle $S^1$, and will also take $X$ to be a copy of $S^1$ (identified with the collection of complex numbers of modulus $1$).  Define $\rho\co X\to \text{Hom}(\Z,U(1))$ by
$$
\rho_x\co n\mapsto x^n.
$$
Concretely, we may identify sections of the line bundle $E_\rho$ over $S^1\times S^1$ with the space of functions $f:\R\times \R\to \C$ that satisfy $f(z+n,x+m)=e^{inx}f(z,x)$ for all $n,m\in\Z$, $z\in\R$ and $x\in \R$.  Now, the formulas
$$
\nabla_z=\partial/\partial z~~,\nabla_x=\partial/\partial x-2\pi i z
$$
define a connection on $E_\rho$, with curvature given by
$$
R(\partial /\partial z,~\partial/\partial x)=2\pi i,
$$
i.e.\ with curvature two-form given by $2\pi idz\wedge dx$ (all of this is just direct computation).  It follows from Chern-Weil theory (see, for example, Milnor--Stasheff \cite[Appendix C]{Milnor-Stasheff}) that the Chern character of $E_\rho$ is given by a generator of $H^2(S^1\times S^1;\R)\cong\R$.  

Now, by Example \ref{flatex}, it suffices to show that any non-zero element of $K_1(B\Z)\otimes \Q\cong \Q$ is flatly detectable.  However, as is well-known, under the Chern isomorphism 
$$Ch: K^0(S^1\times S^1)\otimes \Q\cong H^{even}(S^1\times S^1;\Q),$$
the element $2\pi idz\wedge dx$ corresponds to the element $[u]\otimes [u]\in K^0(S^1\times S^1)\otimes \Q$, where $u:S^1\to U(1)$ is the canonical unitary identifying these spaces, which generates $K^1(S^1)\cong \Z$: up to rational multiples, which is all we need, this follows from the fact that the Chern character is a ring isomorphism, together with the K\"{u}nneth formulas in cohomology and $K$--theory, and the fact that $Ch([u])=[dx]$.  The result follows from this, Lemma \ref{slem} part (2), and (rational) non-degeneracy of the pairing $K_1(S^1)\otimes K^1(S^1)\to \Z$.
\end{proof}

\subsection*{Permanence properties of the class $\FD$}

We now show that $\FD$ is closed under taking free products, direct products and finite index supergroups.

\begin{fpfd}\label{fpfd}
Say $\Gamma_1$ and $\Gamma_2$ are groups in the class $\FD$.  Then their free product $\Gamma=\Gamma_1*\Gamma_2$ is in the class $\FD$.
\end{fpfd}

\begin{proof}
Let $B\Gamma_1$ and $B\Gamma_2$ be finite models for the classifying spaces of $\Gamma_1$ and $\Gamma_2$ respectively, and take $B\Gamma$ to be their wedge sum.  Let $x$ be a non-zero element of $K_*(B\Gamma)\otimes \Q$; we must show $x$ is flatly detectable.

The Mayer-Vietoris sequence in $K$--homology (see Higson--Roe \cite[Lemma 7.1.2]{Higson:2000bs}, Yu~\cite[Proposition 3.11]{Yu:1997} or, for a discussion using the viewpoint on $K$--homology emphasized in the present article, see Roe \cite[pp. 37--38]{Roe:1996}) can now be applied to the
covering of $B\Gamma$ by the subspaces $B\Gamma_1$ and $B\Gamma_2$.  The projections of $B\Gamma$ onto these wedge summands give splittings for the inclusions $B\Gamma_1 \hookrightarrow B\Gamma$ and $B\Gamma_2 \hookrightarrow B\Gamma$, and hence the Mayer--Vietoris sequence gives rise to a natural direct sum decomposition
$$
K_*(B\Gamma)\otimes \Q=(K_*(B\Gamma_1)\otimes \Q)\oplus (\widetilde{K}_*(B\Gamma_2)\otimes \Q)
$$
(and similarly with the roles of $\Gamma_1$ and $\Gamma_2$ reversed).  Without loss of generality, assume that we can write $x=x_1\oplus x_2$ with respect to this decomposition, where $x_1$ is non-zero.  Using the assumption that $\Gamma_1$ is in the class $\FD$, there exists a family of representations $\rho^1\co X\to \text{Rep}(\Gamma_1)$ such that $\rho^1_*(x_1)\neq 0$.  The map $\rho^1$ gives rise to $\rho\co X\to \Rep$ by extending trivially on $\Gamma_2$ and using the universal property of the free product.  

Now, the bundle $E_\rho$ restricted to $B\Gamma_1\times X\subseteq B\Gamma\times X$ is equal to $E_{\rho^1}$ by construction, and is equal to an external product $\C\otimes F$ when restricted to $B\Gamma_2\times X$, where $\C$ is the trivial bundle on $B\Gamma_2$ and $F$ is some bundle over $X$ (trivial on each connected component, but we do not need this).  We then have that
$$
\rho_*(x)=[E_\rho]/x=[E_{\rho^1}]/x_1+(\C\otimes F)/x_2=\rho^1_*(x_1)+\langle \C,x_2\rangle\otimes F
$$
using Lemma \ref{slem}, part (2) and Example \ref{pairex}.  As $x_2$ is an element of the reduced $K$--homology of $B\Gamma_2$, however, $\langle \C,x_2\rangle=0$, whence 
$$
\rho_*(x)=\rho^1_*(x_1)\neq 0
$$
completing the proof.
\end{proof}

More generally, if $\Gamma_1$, $\Gamma_2$ are in the class $\FD$ and $\iota_i\co A\to \Gamma_i$ are split inclusions for $i=1,2$, then the amalgamated free product $\Gamma=\Gamma_1*_A\Gamma_2$ is in $\FD$; this follows from a minor elaboration of the argument above, which is omitted.  It is not true that the class $\FD$ is preserved by arbitrary free products with amalgam: see Example \ref{bmprop} below.

Our next goal is to prove that the class $\FD$ is preserved under direct products, and thus in particular that it contains all finitely generated free abelian groups.  In order to avoid using (somewhat non-trivial - \cite[Chapter 9]{Higson:2000bs}) facts about external products in $K$--homology, the following definition is useful.

\begin{ddef}\label{ddef}
Let $\Gamma$ be a (finitely presented) group with a finite model for the classifying space $B\Gamma$. 

The K\"{u}nneth theorem in $K$--theory (due to Atiyah - see  \cite[Corollary 2.7.15]{Atiyah:1967kn}) implies that the external product induces a natural isomorphism
$$
(K^*(B\Gamma)\otimes \Q)\otimes_\Q (K^*(X)\otimes \Q)\cong K^*(B\Gamma\times X)\otimes \Q.
$$
If $\phi\co K^*(X)\otimes \Q\to \Q$ is any linear functional,  we may thus define a natural map
$$
K^*(B\Gamma\times X)\otimes \Q  \xmaps{1\otimes \phi} (K^*(B\Gamma)\otimes \Q) \otimes_\Q \Q \srm{\isom} K^*(B\Gamma)\otimes \Q,
$$
which by abuse of notation we denote $1\otimes \phi$.
We write $K^*_{\FD}(B\Gamma)$ for the subset of $K^*(B\Gamma)\otimes \Q$ consisting of classes of the form $(1\otimes \phi)[E_\rho]$ where $\rho\co X\to \Rep$ is a family of representations, and $\phi\co K^*(X)\otimes \Q\to \Q$ is a linear functional as above.
\end{ddef}

\begin{dlem}\label{dlem}
With notation as above, $K^*_{\FD}(B\Gamma)$ is a subspace of $K^*(B\Gamma)\otimes \Q$. 

Moreover, these two vector spaces are equal if and only if $\Gamma$ is in the class $\FD$.
\end{dlem}

\begin{proof}
It is clear that $K^*_{\FD}(B\Gamma)$ is closed under scalar multiplication; we will show it is closed under addition.  Let $(1\otimes\phi_i)[E_{\rho^i}]$ be elements of $K^*_{\FD}(B\Gamma)$ for $i=1,2$, where $\rho^i\co X_i\to \Rep$ and $\phi_i\co K^*(X_i)\otimes \Q\to \Q$.  Define
\begin{align*}
\rho\co X_1\sqcup X_2\to \Rep
\end{align*}
by $\rho_x=\rho^i_x$ whenever $x\in X_i$, $i=1,2$.  Then, after identifying $K^*(X_i)\otimes \Q$ with subspaces of $K^*(X_1\sqcup X_2)\otimes \Q$ in the natural way for $i=1,2$, we have $[E_\rho]=[E_{\rho^1}]+[E_{\rho^2}]$, and closure under addition follows from this.  

The remaining claim follows from Lemma \ref{slem} part (2) and rational nondegeneracy of the pairing between $K$--theory and $K$--homology (see for example \cite[Theorem 7.6.1]{Higson:2000bs}).
\end{proof}

\begin{dpfd}\label{dpfd}
Let $\Gamma_1,\Gamma_2$ be in the class $\FD$.  Then the direct product $\Gamma=\Gamma_1\times\Gamma_2$ is in the class $\FD$.
\end{dpfd}

\begin{proof}
Let $B\Gamma_1$ and $B\Gamma_2$ be finite models for the classifying spaces of $\Gamma_1$ and $\Gamma_2$ respectively, and take $B\Gamma$ to be their direct product.

The K\"{u}nneth theorem in $K$--theory \cite[Corollary 2.7.15]{Atiyah:1967kn} implies that the external $K$--theory product induces a natural isomorphism
\begin{equation}\label{kunneth}
(K^*(B\Gamma_1)\otimes \Q)\otimes (K^*(B\Gamma_2)\otimes \Q)\cong K^*(B\Gamma_1\times B\Gamma_2)\otimes \Q
\end{equation}
of graded abelian groups.  Identifying the two sides in line \eqref{kunneth}, Lemma \ref{dlem} implies that it suffices to show that any class  of the form $x_1\otimes x_2$, with $x_i\in K^*(B\Gamma_i)\otimes \Q$ for $i=1,2$, is in $K^*_{\FD}(B\Gamma)$.  

Now, by assumption and Lemma \ref{dlem}, there exist families $\rho^i\co X_i\to \text{Rep}(\Gamma_i)$ and functionals $\phi_i\co K^*(X_i)\otimes \Q\to \Q$ such that $x_i=(1\otimes \phi_i)[E_{\rho^i}]$.  Define a new family $\rho\co X_1\times X_2 \to \Rep$ by ``pointwise tensor product''\footnote{This tensor product map arises from a choice of continuous tensor product map $U(n) \cross U(m) \to U(nm)$; for example one may take the standard Kronecker product of matrices $(A, B) \mapsto A\otimes B$, which commutes with inverses, transposes, and conjugation, and hence maps $U(n)\cross U(m)$ to $U(nm)$.  Since the entries of $A\otimes B$ are just products of entries from $A$ and $B$, continuity is immediate.  Since $(A\otimes B)(C\otimes D) = AC \otimes BD$, this yields a continuous map $\Hom(\Gamma, U(n)) \cross \Hom(\Gamma, U(m)) \to \Hom(\Gamma, U(nm))$.}, i.e.\
$$
\rho\co (x_1,x_2) \mapsto ( ~(g_1,g_2)\mapsto \rho^1_{x_1}(g_1)\otimes \rho^2_{x_2}(g_2)~).
$$

From the construction of $\rho$, it follows that
$$
[E_\rho]=[E_{\rho_1}]\otimes [E_{\rho_2}]\in K^0((X_1\times B\Gamma_1)\times (X_2\times B\Gamma_2))
$$
and thus that (modulo K\"{u}nneth isomorphisms) 
$$
1\otimes \phi_1\otimes \phi_2\co (K^*(B\Gamma)\otimes \Q)\otimes (K^*(X_1)\otimes \Q)\otimes (K^*(X_2)\otimes \Q)\to K^*(B\Gamma)\otimes \Q
$$
takes $[E_\rho]$ to $x_1\otimes x_2$ as required.
\end{proof}

Our next goal is to prove that $\FD$ passes to finite index supergroups; this combines with the previous results to imply, for example, that all torsion free crystallographic groups are in $\FD$.  This requires an analysis of the transfer map in $K$--theory (see for example \cite[Pages 250-1]{Atiyah:1961uq}, where the image of a bundle under transfer is called the \emph{direct image} bundle).  

As an elementary treatment of the $K$--theory transfer seems to be missing from the literature, we give an essentially self-contained account below.  See \cite{Roush:1999fk} for a treatment of transfer in a more general context.  The treatment below is inspired by $KK$-theory, and could be developed completely in that context; we will not do this here.

\begin{transdef}\label{transdef}
Let $Y$ be a finite $CW$ complex with fundamental group $\Gamma$ and universal cover $\widetilde{Y}$.  Let $\Gamma_0$ be a finite index subgroup of $\Gamma$, and $Y_0$ the corresponding finite cover of $Y$.

Note that $Y_0$ is homeomorphic to the balanced product $\widetilde{Y}\times_\Gamma(\Gamma/\Gamma_0)$, whence $C(Y_0)$ is naturally isomorphic to 
$$
\mathcal{T}_Y^{Y_0}:=C_b(\widetilde{Y}\times(\Gamma/\Gamma_0))^\Gamma=C_b(\widetilde{Y}, C(\Gamma/\Gamma_0))^\Gamma.
$$
There is moreover clearly a right $C(Y)=C_b(\widetilde{Y})^\Gamma$ module-structure defined on $\mathcal{T}_Y^{Y_0}$. 

The \emph{transfer map}\footnote{It does of course agree with the more classical notion, as for example in \cite[Pages 7-8]{Roush:1999fk}, but we do not need this.} in $K$--theory, denoted $t\co K^*(Y_0)\to K^*(Y)$ is the homomorphism induced on finitely generated projective modules over $C(Y_0)$ by the formula
$$
E\mapsto E\otimes_{C(Y_0)}\mathcal{T}_Y^{Y_0}.
$$ 
\end{transdef}

The following simple lemma records some properties of the transfer map.

\begin{transprops}\label{transprops}
\begin{enumerate}
\item The $K$--theory transfer is well-defined.
\item Let $\pi:Y_0\to Y$ be a covering map and $t:K^*(Y_0)\to K^*(Y)$ be the corresponding transfer map as in Definition \ref{transdef} above.  Let $E_{\Gamma_0}$ be the ``flat''\footnote{$Y$ may not be a manifold, so this does not literally make sense.} bundle over $Y$ induced by the quasi-regular representation of $\Gamma$ on $l^2(\Gamma/\Gamma_0)$.  Then 
the composition 
$$
t\circ \pi^*\co K^*(Y)\to K^*(Y)
$$
is equal to the (internal) $K$--theory product with $[E_{\Gamma_0}]\in K^*(Y)$.  

In particular, $t\circ\pi^*$ is a rational isomorphism. 
\item With notation as in part (2), let $\rho:X\to \text{Rep}(\Gamma_0)$ be a family of representations.  Let 
$$
Ind(\rho):X\to \Rep
$$
be the family defined by ``pointwise induction''.\footnote{Pointwise induction arises from a choice of continuous induction map $\Hom(\Gamma_0, U(n)) \to \Hom(\Gamma, U(n[\Gamma : \Gamma_0]))$.  This map depends on a choice of coset representatives for $\Gamma_0$ in $\Gamma$.  A detailed discussion can be found in~\cite{Ramras-crystal}}
Then if $E_{\rho}$, $E_{Ind(\rho)}$ are the bundles over $Y_0\times X$ and $Y\times X$ defined by $\rho$, $Ind(\rho)$ respectively, we have that 
$$
t[E_\rho]=[E_{Ind(\rho)}]\in K^*(Y).
$$
\item For any finite covering $Y_0 \to Y$ and any space $X$, the transfer map for the product covering $Y_0 \cross X \to Y_0 \cross X$ satisfies $t (y \cross x) = t(y) \cross x$ for all $y\in K^* (Y_0)$ and $x\in K^* (X)$. 
\end{enumerate}
\end{transprops}
\begin{proof}
\begin{enumerate}
\item For the case of $K^0$, this follows from the fact that $\mathcal{T}_Y^{Y_0}$ is finitely generated and projective, both as a left $C(Y_0)$ module and as a right $C(Y)$ module.  The first of these is obvious - it is a free rank one module over $C(Y_0)$ - while the second follows from the fact that it is equal as a $C(Y)$ module to the sections of the bundle over $Y$ induced by the representation of $\Gamma$ on $l^2(\Gamma/\Gamma_0)$.  The case of higher $K$--groups can be considered by taking suspensions (this is probably most easily seen with the ``analysts suspension'' -- taking the tensor product with $C_0(\R)$, and using $K$--theory with compact supports).
\item The homomorphism $\pi^*:C(Y)\to C(Y_0)$ induces a left $C(Y)$ module structure on $C(Y_0)$; write $\Pi$ for the corresponding $C(Y)$--$C(Y_0)$ module.  The composition $t\circ \pi^*$ is then equal to the map on $K^*(Y)$ induced by taking tensor product with the (finitely generated, projective) $C(Y)$ module
$$
\Pi\otimes _{C(Y_0)}\mathcal{T}_Y^{Y_0};
$$
this module simply \emph{is} the sections of $E_{\Gamma_0}$, however.  

The remaining statement follows from Chern-Weil theory in case $Y$ and $Y_0$ are manifolds (not necessarily closed): indeed, rationally, taking the product with $E_{\Gamma_0}$ is simply multiplication by $|\Gamma/\Gamma_0|$.  The general case follows on replacing $Y$ by a homotopy equivalent  manifold (which need not be closed) and $Y_0$ with the corresponding cover.
\item  Assume for simplicity of notation that $X$ is connected (for the general case, consider each connected component separately), in which case we may assume that under each $\rho_x$, $\Gamma$ acts on some fixed $\C^k$.  The space of sections of $E_\rho$ is then given by
$$
C_b(\widetilde{Y}\times X,\C^k)^{\Gamma_0},
$$
(where the fixed points are taken for a $\Gamma_0$ action analogous to the $\Gamma$ action in Definition \ref{bundef}) while that for $E_{Ind(\rho)}$ is given by
$$
C_b(\widetilde{Y}\times X,C_b(\Gamma,\C^k)^{\Gamma_0})^\Gamma;
$$
it is not difficult to see that tensoring the former over $C(Y_0)$ by $\mathcal{T}_Y^{Y_0}$ yields the latter, which is the claim.

\item Let $E$, $F$ be finitely generated projective modules over $C(Y_0)$, $C(X)$ respectively (i.e.\ spaces of sections of bundles over the respective spaces).  It suffices to show that 
$$
(E\otimes F)\otimes_{C(X\times Y_0)} \mathcal{T}_{Y\times X}^{Y_0\times X}\cong (E\otimes_{C(Y_0)}\mathcal{T}_Y^{Y_0})\otimes F,
$$
which is a straightforward computation.
\end{enumerate}
\end{proof}

\begin{fifd}\label{fifd}
Say $\Gamma$ is a group with finite classifying space, and that $\Gamma_0$ is a finite index subgroup of $\Gamma$ in the class $\FD$.  Then $\Gamma$ is in the class $\FD$.
\end{fifd}

\begin{proof}
Let $B\Gamma$ be a finite classifying space for $\Gamma$, and $B\Gamma_0$ the finite cover of $B\Gamma$ corresponding to the inclusion $\Gamma_0\hookrightarrow\Gamma$, which is a finite classifying space for $\Gamma_0$.  Let $x$ be an element of $K^*(B\Gamma)\otimes\Q$; by Lemma \ref{dlem}, it suffices to show that $x$ is in $K^*_{\FD}(B\Gamma)$.  Now, by part (2) of Lemma \ref{transprops}, 
$$
t:K^*(B\Gamma_0)\otimes \Q\to K^*(B\Gamma)\otimes\Q
$$
is surjective (where, as usual, we have abused notation, writing `$t$' for `$t\otimes Id_\Q$'), whence there exists $y\in K^*(B\Gamma_0)\otimes \Q$ with $t(y)=x$.  As $\Gamma_0$ is in the class $\FD$, and by Lemma \ref{dlem} again, there exist $\rho:X\to \text{Rep}(\Gamma_0)$ and $\phi:K^*(X)\otimes \Q\to \Q$ such that $y=(1\otimes \phi)[E_\rho]$.  To complete the proof, note that the diagram 
$$
\xymatrix {K^*(B\Gamma_0\times X)\otimes \Q\ar[r]^t \ar[d]^{1\otimes \phi} & K^*(B\Gamma\times X)\otimes \Q \ar[d]^{1\otimes \phi} \\
K^*(B\Gamma_0)\otimes \Q \ar[r]^t & K^*(B\Gamma)\otimes \Q }
$$ 
commutes by part (4) of Lemma~\ref{transprops}, whence using part (3) of Lemma \ref{transprops}
$$
x=t(y)=t((1\otimes \phi)[E_\rho])=(1\otimes \phi)(t[E_\rho])=(1\otimes \phi)[E_{Ind(\rho)}]
$$
and we are done.
\end{proof}

\subsection*{Surface groups}

Our next goal is to show that fundamental groups of compact, aspherical surfaces are in $\FD$.  This will rely on Yang--Mills theory, and in particular on a result of the first author from~\cite{Ramras-stable-moduli}.  To begin, we need to analyze the classifying map for the bundle $E_\rho$ associated to a family of representations.  In order to do this, we will need to consider a functorial model $\B (-)$ for classifying spaces, e.g. Milnor's infinite join construction~\cite{Milnor} or Segal's simplicial model~\cite{Segal}.  These have the property that homomorphisms $\rho\co G\to H$ induce continuous maps $\B(\rho) \co \B G \to \B H$, and in fact if $G$ and $H$ are topological groups, this gives rise to a continuous map
$$\Hom(G, H) \srm{\B} \Map_* (\B G, \B H).$$
(Continuity of this map is most easily checked using Segal's model, which gives a model for the classifying space so long as $G$ and $H$ are Lie groups, which suffices for our purposes.)
A continuous map $\rho\co X\to \Hom(G, H)$ now has an associated map $\B \circ \rho\co X\to \Map_* (\B G, \B H)$, which has an adjoint $X\cross \B G \to \B H$ (this adjoint is continuous so long as $X$ is locally compact and Hausdorff, e.g. if $X$ is a CW complex).  We will denote this adjoint by $\rho^\vee$.
 The functorial model $\B (-)$ has an associated functorial model $\E (-)$ for the universal bundle, so that $\E G \to \B H$ is a universal  (left) principal $G$--bundle.  Moreover, there is a continuous mapping $\Hom(G, H) \srm{\E} \Map (\E G, \E H)$, such that for any $\rho\co G\to H$, $\E (\rho)\co \E G\to \E H$ is $\rho$--equivariant in the sense that $\E (\rho) (g\cdot e) = \rho(g) \cdot \E(\rho) (e)$.  Moreover, the diagram
\begin{equation}\label{Erho-Brho}
\xymatrix{ \E G \ar[d] \ar[r]^-{\E(\rho)} & \E H\ar[d]\\
		\B G \ar[r]^-{\B(\rho)} & \B H}
\end{equation}
commutes for each $\rho\co G\to H$.
 

\begin{classifying-map}$\label{classifying-map}$
Let $\rho \co X\to \Hom(\Gamma, U(n))$ be an $X$--family of representations.  Let $B\Gamma$ be a finite model for the classifying space of $\Gamma$, and let $f\co B\Gamma\to \B \Gamma$ be a classifying map for the universal $\Gamma$--bundle $E\Gamma\to B\Gamma$.  Then the composite map 
$$B\Gamma \cross X \xmaps{f\cross \Id_X} \B \Gamma \cross X \srm{\rho^\vee} \B U(n)$$
is a classifying map for the principal $U(n)$--bundle associated to $E_\rho$.
\end{classifying-map}
\begin{proof}
The principal $U(n)$--bundle associated to $E_\rho$ is simply 
$$\xymatrix{(E\Gamma \cross X \cross U(n))/\Gamma \ar[d] \\ B\Gamma \cross X},$$ 
where $\Gamma$ acts by $g \cdot (e, x, A) = (g \cdot e, x,  \rho_x (g) A)$.  This can be viewed as a \e{left} principal $U(n)$--bundle, via the action $A \cdot [e, x, B] = [e, x, BA^{-1}]$.  There is then an analogous left principal $U(n)$--bundle over $\B\Gamma \cross X$, formed by replacing $E\Gamma$ with $\E\Gamma$ in the previous construction.
We will construct a commutative diagram of left principal $U(n)$--bundles as follows:
$$\xymatrix{
(E\Gamma \cross X \cross U(n))/\Gamma \ar[rr]^{\tilde{f} \cross \Id \cross \Id} \ar[d] &
	& (\E\Gamma \cross X \cross U(n))/\Gamma \ar[d] \ar[r]^-\alpha
	& \E U(n) \ar[d]\\
B\Gamma \cross X \ar[rr]^-{f\cross \Id} & & \B \Gamma \cross X \ar[r]^-{\rho^\vee} & \B U(n)
}$$
The map $\tilde{f}\co E\Gamma\to \E \Gamma$ is the unique map of principal bundles covering $f$, and hence the left-hand square commutes by construction.  Moreover, $\tilde{f}\cross \Id\cross \Id$ induces a $U(n)$--equivariant map between these quotient spaces, and hence the left-hand square is a pull-back diagram of principal $U(n)$--bundles.  The map $\alpha$ is defined by
$$\alpha ([e, x, A]) = A^{-1}\cdot \E(\rho_x) (e).$$
It follows from the properties of $\E(\rho_x)$ listed above that this is a $U(n)$--equivariant map, and that the right-hand diagram commutes.  Thus the right-hand square is also a pullback diagram of principal $U(n)$--bundles, completing the proof.
\end{proof}



The next technical-but-simple lemma comes down to the relationship of the `analysts suspension' $X\times \R$ and the `topologists suspension' $X\sm S^1$, and a way of making sense of the slant product on reduced $K$--theory and $K$--homology.

\begin{pairing-lemma}\label{pairing-lemma}
Let $\Gamma$ be a group with a finite model $B\Gamma$ for its classifying space.  Let $x$ be an element of $\wt{K}_i(B\Gamma)$ which is non-zero after tensoring with $\Q$.  Then for each $k\geq 0$ there exists 
$$
y_k\in \wt{K}^i (B\Gamma\sm S^{2k+i})=\wt{K}^0(B\Gamma\sm S^{2k})
$$ 
such that if 
$$
\pi:B\Gamma\times S^{2k+i}\to B\Gamma\sm S^{2k+i}
$$ 
is the natural quotient map then the slant product $\pi^*(y_k)/x$ is a well-defined element of $K^*(S^{2k+i})$ and is non-zero.
\end{pairing-lemma}

\begin{proof} 
Let $x_0\in B\Gamma$ and $\infty\in S^{2k+i}$ be the respective basepoints, and identify $S^{2k+i}\backslash \{\infty\}$ with $\R^{2k+i}$.  Recall that the $K$--theory (respectively, $K$--homology) of a locally compact, non-compact space $Y$ is identified with the reduced $K$--theory (resp.\ $K$--homology) of the one point compactification $Y^+$, which is in turn identified with $K_*(C_0(Y))$ (resp.\ $K^*(C_0(Y))$).  The statement of the lemma can thus be rewritten as follows: for any rationally non-trivial element $x\in K_i(C_0(B\Gamma\backslash\{x_0\}))$ and any $k\geq 0$ there exists 
$$
y_k\in K^i(C_0(B\Gamma\backslash \{x_0\})\otimes C_0(\R^{2k+i}))
$$ 
such that if 
$$
\iota:C_0(B\Gamma\backslash \{x_0\})\otimes C_0(\R^{2k+i})\to C_0(B\Gamma\times S^{2k+i}\backslash \{(x_0,\infty)\})
$$ 
is the natural inclusion then 
$$
0\neq \iota_*(y_k)/x\in K^0(S^{2k+i})
$$
(here we think of $K_*(C_0(B\Gamma\times S^{2k+i}\backslash \{(x_0,\infty)\}))$ as a subring of $K_*(C(B\Gamma\times S^{2k+i}))$ to make sense of the slant product in the above). 

This is not difficult, however: take any element $z\in K^i(C_0(B\Gamma))$ such that $x/z=\langle x,z\rangle$ is non-zero (which exists by rational non-degeneracy of the pairing), let $y_k=z\otimes b$, $b\in K^0(\R^{2k+i})$ the Bott generator, and apply Lemma \ref{slem} part 2.
\end{proof}

\begin{B-surj-fd-thm}\label{B-surj-fd-thm}
Let $\Gamma$ be a group with a finite model $B\Gamma$ for its classifying space.  Say there exists $K>0$ such that 
for each $k>K$, there exists $N=N(k)>0$ such that for $n>N$,
the natural map
$$\pi_k \Hom(\Gamma, U(n)) \srm{\B_*} \pi_k \Map_* (\B\Gamma, \B U(n))$$
is surjective.  Then $\Gamma\in \FD$.
\end{B-surj-fd-thm}

\begin{proof}  We need to show that each $K$--homology class on $B\Gamma$ is flatly detectable.  Since the unit element is detected by the trivial representation, it will suffice to work with reduced $K$--homology.  
By Lemma~\ref{pairing-lemma}, we know that for each rationally non-zero $x\in \wt{K}_i (B\Gamma)$ and each $k>0$, there exists $y_k\in \wt{K}^0 (B\Gamma\sm S^{2k+i})$ such that $\pi^*(y_k)/x$ is non-zero, where $\pi$ is the quotient map 
$B\Gamma \cross S^{2k+i} \to B\Gamma \sm S^{2k+i}$.

Choose $k$ large enough that $2k+i> K$, and let $N = N(2k+i)$ be the number guaranteed by the hypothesis.  Now $y_k$ has the form $y_k = [V] - [W]$ for some bundles $V, W$ over $B\Gamma \sm S^{2k+i}$.  The bundles $V$ and $W$ are then classified by maps $\alpha_V, \alpha_W \co B\Gamma \sm S^{2k+i} \to \B U(n)$ for some $n$, and we may assume that $n>N$.  Moreover, we can assume these maps are based, and hence correspond to classes 
$\alpha_V, \alpha_W \in \pi_{2k+i} \Map_* (B\Gamma, \B U(n))$.  Choose a classifying map $f\co B\Gamma\to \B\Gamma$ for the universal bundle $E\Gamma\to B\Gamma$,  and a homotopy inverse $g\co \B\Gamma \to B\Gamma$ (note that $g$ classifies $\E \Gamma$).  By abuse of notation, $g$ will also denote the induced map $\B\Gamma \sm S^{2k+i} \to B\Gamma \sm S^{2k+i}$ and the map $\Map_* (B\Gamma, \B U(n)) \to \Map_*(\B\Gamma, \B U(n))$ induced by pre-composition with $g$.
Then $g^* V$ and $g^*W$ are classified by the adjoints of the elements $g_* \alpha_V$ and $g_* \alpha_W$ (respectively).

Our hypothesis now yields classes $\rho_V, \rho_W \in \pi_{2k+i} \Hom(\Gamma, U(n))$ 
such that $\B_* \rho_V = g_*\alpha_V$ and $\B_* \rho_W= g_*\alpha_W$.  By Lemma~\ref{classifying-map} the bundle $E_{\rho_V}$ is classified by the map 
$$B\Gamma \cross S^{2k+i} \xmaps{f\cross \Id} \B\Gamma \cross S^{2k+i} \xmaps{\rho_V^\vee} \B U(n).$$
By definition, $\rho_V^\vee = (b, z) = \alpha_V (z) (g(b))$.  So in fact, $E_{\rho_V}$ is classified by 
$$B\Gamma \cross S^{2k+i} \xmaps{f\cross \Id} \B\Gamma \cross S^{2k+i} \xmaps{g\cross \Id} B\Gamma \cross S^{2k+i}
\xmaps{\pi} B\Gamma \sm S^{2k+i} \srm{\alpha_V} \B U(n)$$
Since $g\circ f$ is homotopic to the identity, we conclude that
 $E_{\rho_V} \isom \pi^* V$.  Similarly, we have  $E_{\rho_W} \isom \pi^* W$.

Since $y_k = ([V] - [W])$, we have 
$$0\neq (\pi^*y_k) / x = ([\pi^* V] - [\pi^* W])/x = ([E_{\rho_V}] - [E_{\rho_W}])/x,$$ 
so we must have either 
$(\rho_V)_* (x) = [E_{\rho_V}]/x  \neq 0$ or $(\rho_W)_* (x) = [E_{\rho_W}]/x  \neq 0$; in either case we conclude that $x$ is flatly detectable.
\end{proof}

\begin{sgfd}\label{sgfd}
Let $M^2$ be a compact, aspherical surface (possibly with boundary).  Then $\pi_1 M^2 \in \FD$.
\end{sgfd}
\begin{proof}
If $\partial M^2\neq \emptyset$, then $\pi_1 M^2$ is isomorphic to a finitely generated free group $F_m$.  This case follows from Proposition \ref{zfd} and Lemma \ref{fpfd} above, but we give a different proof here for the sake of variety.

Note then that the natural map
$$\Hom(F_m, U(n)) = U(n)^m \maps \Map_* (\B F_m, \B U(n))$$
is a weak equivalence for each $n$: this map can be identified with the natural weak equivalence $U(n)^m \heq (\Omega \B U(n))^m = \Map_* (\bigvee_m S^1, \B U(n))$, using the fact that $\B F_m \heq \bigvee_m S^1$.  For further details, see Ramras~\cite[Proof of Theorem 4.3]{Ramras-stable-moduli}.  Thus $F_m$ satisfies the hypotheses of Theorem~\ref{B-surj-fd-thm}.

The case of closed, aspherical surfaces follows from Theorem~\ref{B-surj-fd-thm} together with~\cite[Theorem 3.4]{Ramras-stable-moduli}, which states that the natural map
$$\Hom(\pi_1 M^2, U(n))   \maps \Map_* (\B \pi_1 M^2, \B U(n))$$
induces an isomorphism on homotopy groups in dimensions $0<* < n-1$.  
\end{proof}

\subsection*{Closing remarks on $\FD$}

To complete this section, we discuss the class $\FD$ a little more broadly, and give an application to computing the Betti numbers of representation varieties.

For the readers' convenience, the next corollary summarizes what are perhaps the most natural examples we know to be in the class $\FD$.

\begin{fdex}\label{fdex}
The following classes of groups are in the class $\FD$:
\begin{itemize}
\item Finitely generated free groups.
\item Finitely generated free abelian groups.
\item Torsion free crystallographic groups.
\item Fundamental groups of compact, aspherical surfaces.  \qed
\end{itemize}
\end{fdex}

On the other hand, the class $\FD$ seems likely to be quite restrictive.  The next result gives an explicit example of some groups that are not in the class $\FD$, and in particular shows that $\FD$ is not closed under free products with amalgam.  

Recall first that Burger and Mozes \cite{Burger:2000vn} have shown that there exist (infinitely many) groups $\Gamma$ with the following three properties.
\begin{enumerate}
\item $\Gamma$ is equal to a free product with amalgam $F*_GF$, where $F$ and $G$ are non-abelian finitely generated free groups, and $G$ is embedded in $F$ as a finite index subgroup in two different ways.
\item There is a classifying space for $\Gamma$ which is a two-dimensional finite $CW$ complex.
\item $\Gamma$ is simple.
\end{enumerate}

\begin{bmprop}\label{bmprop}
Let $\Gamma$ be one of the groups constructed by Burger and Mozes with the properties listed 1, 2, 3 above.   Then $\Gamma$ is not in $\FD$. 
\end{bmprop}

\begin{proof}
Infinite, finitely generated, simple groups have no non-trivial finite-dimensional representations (due to the fact that finitely generated linear groups are residually finite), whence the only flatly detectable classes in $K_*(B\Gamma)$ are multiples of the unit class.  It thus suffices to show that $\wt{K}_0(B\Gamma)\otimes \Q$ is non-trivial.

Consider the Mayer-Vietoris sequence associated to a free product decomposition in point 1 above \cite[Corollary 7.7]{Brown:1982kx}, part of which is
$$
0\to H_2(\Gamma;\Q)\to H_1(G;\Q)\to H_1(F;\Q)\oplus H_1(F;\Q) \to\cdots
$$
Let $g$ denote the rank of $G$ and $f$ the rank of $F$, so the above gives a sequence 
$$
0\to H_2(\Gamma;\Q)\to \Q^g\to \Q^f\oplus \Q^f \to\cdots,
$$
where moreover the map $\Q^g\to \Q^f\oplus \Q^f$ is of the form $\phi\oplus\psi$ for some maps $\phi,\psi:\Q^g\to \Q^f$.  
Multiplicativity of Euler characteristics under taking finite (graph) covers implies that 
$$
1-g=[F:G](1-f),~~\text{i.e.} ~~g=[F:G](f-1)+1
$$
(although two different embeddings of $G$ into $F$ are used, the two images of $G$ in $F$ necessarily have the same index - this follows by the above formula - so the notation $[F:G]$ is unambiguous). As the index $[F:G]$ and the rank $f$ are both larger than $1$, this forces $g>f$.  It follows that the map
$$
\phi\oplus\psi:\Q^g\to \Q^f\oplus \Q^f
$$
cannot be injective, and thus that $H_2(B\Gamma;\Q)=H_2(\Gamma;\Q)$ is not-zero.  The Chern character isomorphism now implies that $\wt{K}_0(B\Gamma)\otimes \Q$ is non-trivial, and we are done.
\end{proof}

Note that the strong Novikov conjecture is certainly true for $\Gamma$ as above, however, e.g.\ by using the results of \cite{Yu:1998wj}.  Other examples similar to the groups of Burger and Mozes could be extracted from work of Wise \cite{Wise:1995kx}, and another source of similar examples is explained in Ramras~\cite[Section 2.1]{Ramras-surface}.

We suspect the following groups are also not in the class $\FD$, although we were unable to prove this (and would be happy to be proved wrong!).

\begin{fdnex}\label{fdnex}
Are the following groups in the class $\FD$?
\begin{itemize}
\item The integral Heisenberg group.
\item Infinite property (T) groups.
\end{itemize}
\end{fdnex}

The case of property (T) groups seems plausible as any family of finite dimensional representations (parametrized by a connected space) consists of representations that are all mutually equivalent, as is essentially contained in \cite[Theorem 1.2.5]{Bekka:2000kx}.  The reason this does not yield a proof is that it does not preclude the existence of interesting topology \emph{within} each equivalence class of  representations.

The questions below seem natural and interesting.

\begin{fdq}\label{fdq}
\begin{itemize}
\item It follows from Proposition \ref{bmprop} that there exist free products with amalgam $\Gamma=\Gamma_1*_A\Gamma_2$ such that $\Gamma_1,\Gamma_2,A$ are all in $\FD$, but $\Gamma$ is not.  Are there `reasonable' conditions on a free product with amalgam which imply it is in $\FD$?
\item Which torsion free one relator groups are in $\FD$ (possibly all)?
\item Which three-manifold groups are in $\FD$ (possibly all)?
\item One could define a larger version of the class $\FD$ by considering families of \emph{quasi-representations}, i.e.\ maps $\Gamma\to U(n)$ that agree with a homomorphism on a given set of generators up to some $\epsilon$.  What sort of groups have this weaker property (which should also imply the strong Novikov conjecture)?  Recent work of Dadarlat \cite{Dadarlat:2011uq,Dadarlat:2011kx}\footnote{We would like to thank Marius Dadarlat for sharing these preprints with us.} investigating assembly and quasi-representations (among other things) seems very relevant here.
\end{itemize}
\end{fdq}

We end this section by noting some consequences of our results for the topology of unitary representation varieties.  The identity map $\Id$ on $\Hom(\Gamma, U(n))$ may be viewed as the \e{universal} $n$--dimensional family of representations, and we denote the associated bundle by $\mathcal{U} = E_\Id \to \Hom(\Gamma, U(n)) \cross B\Gamma$.  

\begin{universal} \label{universal}  If $\Gamma$ is in $\FD$, then for sufficiently large $n$ there is a rationally injective map
$$K_* (B\Gamma) \maps K^* (\Hom(\Gamma, U(n)))$$
given by $x\mapsto \mathcal{U}/x$.  Consequently, the sum of the (even, or odd) Betti numbers of $\Hom(\Gamma, U(n))$ is at least that of $\Gamma$.
\end{universal}
\begin{proof} Since $\Gamma$ is in $\FD$, every rational $K$--homology class $x\in K_* (B\Gamma) \otimes \Q$ satisfies $\rho_*(x) \neq 0\in K^* (X)$ for some family of representations $\rho\co X\to \Hom(\Gamma, U(n))$.  By Part 1 of Lemma~\ref{slem}, we have $\rho^* ([\mathcal{U}]/x) = [E_\rho]/x = \rho_* (x)$, so $[\mathcal{U}]/x\in K^* (\Hom(\Gamma, U(n)))\otimes \Q$ must be non-zero as well.  As we have assumed $B\Gamma$ has the homotopy type of a finite CW complex, $K_* (B\Gamma)\otimes \Q$ is finitely generated, so any sufficiently large $n$ works for all $x\in K_* (B\Gamma)$ (note here that $E_{\rho\oplus 1} = E_\rho \oplus E_1$, so $(\rho\oplus 1)_* (x) = \rho_* (x) + 1_* (x) = \rho_* (x)$ for any $x\in \wt{K}_* (B\Gamma)$).
The statement about cohomology follows from consideration of the Chern character.
\end{proof}

\section{Families of representations and analytic assembly}\label{togsec}

In this section, we relate groups in the class $\FD$ from Section \ref{fdsec} back to the analytic assembly map from Section \ref{asssec}.  The main result is as follows.

\begin{assfac}\label{assfac}
For each family of representations $\rho\co X \to \text{Hom}(\Gamma, U(k))$, the detecting map $\rho_*\co K_*(B\Gamma)\to K^*(X)$ in Definition \ref{bunhom} factors through the analytic assembly map
$$
\mu\co K_*(B\Gamma)\to K_*(C^*(\Gamma))
$$
defined in Definition \ref{ass}.
\end{assfac}

To prove this, we will need a lemma relating the Miscenko bundle $M_\Gamma$ and the bundle $E_\rho$ associated to $\rho$.  Note first that if $\rho\co X\to \text{Rep}_k(\Gamma)$ is a family of representations, then $\rho$ defines a $*$-homomorphism
\begin{align*}
\rho^\sharp\co C^*(\Gamma) & \to M_k(C(X)) \\
u_g & \mapsto (~x\mapsto \rho_x(g)~).
\end{align*}

\begin{rholem}\label{rholem}
The image of the Miscenko line bundle $[M_\Gamma]\in K_0(C^*(\Gamma)\otimes C(B\Gamma))$ under the map induced by the $*$-homomorphism
$$
\rho^\sharp\otimes 1_{C(B\Gamma)}:C^*(\Gamma)\otimes C(B\Gamma)\to M_k(C(X))\otimes  C(B\Gamma)
$$
identifies naturally with the class $[E_\rho]\in K^0(B\Gamma\times X)$ from Definition \ref{bundef} above.
\end{rholem}

\begin{proof}
Recall from the proof of Lemma \ref{misbunlem} that the space of sections of the Miscenko bundle identifies naturally with $C_b(E\Gamma,C^*(\Gamma))^\Gamma$.   It follows that the image of $[M_\Gamma]$ under $\rho^\sharp\otimes 1_{C(B\Gamma)}$ is the class in $K_*(C(B\Gamma\times X,M_k(\C)))$ of the module
$$
C_b(E\Gamma,C^*(\Gamma))^\Gamma\bigotimes_{C(B\Gamma,C^*(\Gamma))}C(B\Gamma\times X,M_k(\C)),
$$ 
where we have used the natural isomorphism $C(B\Gamma, M_k(C(X)))\cong C(B\Gamma\times X,M_k(\C))$, and the tensor product is defined via the left action of $C(B\Gamma,C^*(\Gamma))$ on $C(B\Gamma\times X,M_k(\C))$ coming from $\rho^\sharp\otimes 1_{C(B\Gamma)}$.  Define now a $\Gamma$ action on $C_b(E\Gamma\times X,M_k(\C))$ by
$$
(g\cdot f)(z,x):=\rho_x(g)f(g^{-1}z,x),
$$
and let $C_b(E\Gamma\times X,M_k(\C))^\Gamma$ denote the fixed points.  There is an isomorphism of $C(B\Gamma\times X,M_k(\C))$ modules
\begin{align*}
C_b(E\Gamma,C^*(\Gamma))^\Gamma\bigotimes_{C(B\Gamma,C^*(\Gamma))}C(B\Gamma\times X,M_k(\C)) \stackrel{\cong}{\to} C_b(E\Gamma\times X,M_k(\C))^\Gamma 
\end{align*}
defined for $f\in C_b(E\Gamma,C^*(\Gamma))^\Gamma$ and $h\in C(B\Gamma\times X,M_k(\C))$ by 
\begin{align*}
f\otimes h & \mapsto (~(z,x)\mapsto \rho_x(f(z))h(\pi(z),x )~)
\end{align*}
(where we have extended $\rho_x$ from $\Gamma$ to $C^*(\Gamma)$).
Hence the image of $[M_\Gamma]$ in $K_0(C(B\Gamma\times X,M_k(\C)))$ is represented by the finitely generated projective module 
$$
C_b(E\Gamma\times X,M_k(\C))^\Gamma; 
$$
the essential point is that this is the space of sections of the endomorphism bundle of $E_\rho$, which completes the proof.  More concretely, the image of this module under the Morita equivalence isomorphism
$$
K_*(C(B\Gamma\times X,M_k(\C)))\cong K^*(B\Gamma\times X)
$$
is given by
$$
C_b(E\Gamma\times X,M_k(\C))^\Gamma\bigotimes_{C(B\Gamma\times X,M_k(\C))}C(B\Gamma\times X,\C^k).
$$
Define a $\Gamma$ action on $C_b(E\Gamma\times X,\C^k)$ by 
$$
(g\cdot f)(z,x):=\rho_x(g)f(g^{-1}z,x);
$$
then there is an isomorphism
\begin{align*}
C_b(E\Gamma\times X,M_k(\C))^\Gamma\bigotimes_{C(B\Gamma\times X,M_k(\C))} & C(B\Gamma\times X,\C^k) \stackrel{\cong}{\to} C_b(E\Gamma\times X, \C^k)^\Gamma,
\end{align*}
defined for $f\in C_b(E\Gamma\times X,M_k(\C))^\Gamma$ and $h\in C(B\Gamma\times X,\C^k)$ by 
$$
f \otimes h \mapsto (~(z,x)\mapsto f(z,x) h(\pi(z),x)~)
$$
(here $\pi\co E\Gamma\to B\Gamma$ is the canonical quotient).
However, $C_b(E\Gamma\times X,\C^k)^\Gamma$ is simply the space of sections of $E_\rho$, and we are done.
\end{proof}
 
\begin{proof}[Proof of Proposition \ref{assfac}]
Consider a family of representations 
$$\rho\co X\to \Rep.$$ 
We will show that the following diagram commutes:
\begin{equation}\label{factorization}\xymatrix{ K_* (B\Gamma) \ar[r]^-\mu \ar[dr]_-{\rho_*} & K_* (C^*(\Gamma)) \ar[d]^-{(\rho^\sharp)_*}\\
	 & K^* (X).}
\end{equation}
Using Definition \ref{ass}, and Lemmas \ref{slem} and \ref{rholem}, we have that if 
$$[M_\Gamma]\in K_0(C(B\Gamma)\otimes C^*(\Gamma))$$
 is the Miscenko bundle and $x\in K_*(B\Gamma)$, then
$$
((\rho^\sharp)_*\circ\mu)(x)=(\rho^\sharp)_*([M_\Gamma]/x)=((\rho^\sharp\otimes 1_{C(B\Gamma)})_*[M_\Gamma])/x=[E_\rho]/x=\rho_*(x),
$$
where $\rho_*$ is an is Definition \ref{bunhom}.  The result follows.
\end{proof}

\begin{injcor}\label{injcor}
Let $\Gamma$ be group with finite classifying space $B\Gamma$ and let $x\in K_*(B\Gamma)\otimes \Q$ be a flatly detectable class.  Then 
$$(\mu\otimes\text{Id}_\Q)(x)\in K_*(C^*(\Gamma))\otimes\Q$$ 
is non-zero.  

In particular, if $\Gamma$ is a group in the class $\FD$, then the analytic assembly map is rationally injective.  \qed
\end{injcor}

\begin{proof}  If $x\in K_i(B\Gamma)\otimes \Q$ is flatly detectable, then (Definition~\ref{fddef}) 
there exists a family $\rho\co X\to \Rep$ such that $(\rho_*\otimes \Id_\Q) (x)$ is non-zero in
$K^{-i}(X)\otimes \Q$.  Commutativity of Diagram (\ref{factorization}) now shows that 
$(\mu\otimes\text{Id}_\Q)(x)$ is non-zero as well.  If $\Gamma$ is in the class $\FD$, then every class in $K_*(B\Gamma)\otimes \Q$ is flatly detectable, so we find that 
the kernel of $\mu\otimes\text{Id}_\Q$ is trivial, as desired.
\end{proof}

Corollary~\ref{injcor} is enough to imply, for example, the Novikov conjecture for $\Gamma$ (a proof may be found in \cite{Kaminker:1985fk}; note that the argument there is phrased in terms of the reduced $C^*$--algebra of $\Gamma$, but also applies to $C^*(\Gamma)$), and that (if a closed manifold) $B\Gamma$ does not admit a metric of positive scalar curvature (see \cite{Rosenberg:1983qy}).

\appendix

\section{The slant product and the Kasparov product}\label{kasap}

In the main part of the piece (Definition \ref{slant}), we have introduced a slant product in operator $K$--theory in order to give an elementary picture of the assembly map.  The more usual way of describing the assembly map is via the \emph{Kasparov product}: see for example \cite[Section 6]{Kasparov:1988dw}, where the assembly map is denoted $\beta$.  In this appendix, we show that our slant product agrees with the Kasparov product.  More precisely, we prove the following result.

\begin{slakas}\label{slakas}
The slant product 
$$
K_i(A\otimes B)\otimes K^j(B)\to K_{i-j}(A)
$$
from Definition \ref{slant} agrees naturally with the Kasparov product
$$
KK_i(\C,A\otimes B)\otimes KK^j(B,\C)\to KK_{i-j}(\C,A).
$$
\end{slakas} 

\begin{proof}
For the sake of simplicity (and as it is the only case we need), assume that $A$ and $B$ are unital $C^*$-algebras.  Using a suspension argument, it suffices to consider the case $i=j=0$.

It suffices to show that if $p\in M_n(A\otimes B)$ defines a class $[p]\in K_0(A\otimes B)$ and $u\in\mathcal{D}(B)$ defines a class $[u]\in K^0(B)$, then (under the natural identifications of these groups with the corresponding $KK$-groups) the two products agree.  We start by constructing $KK$-elements corresponding to $p$ and $u$; we use the standard Kasparov picture of $KK$ (with graded formalism).
\begin{itemize}
\item Let $l^2$ denote the Hilbert space $l^2(\N)$. We may consider $p$ as defining a bounded operator on the Hilbert-($A\otimes B$)-module $A\otimes B\otimes l^2$, which is supported on 
$$
A\otimes B\otimes \text{span}\{\delta_0,...,\delta_{n-1}\}.
$$
Using Kasaprov's stabilization theorem \cite[Page 151]{Kasparov:1988dw}, there exists a partial isometry $v\in M(A\otimes B\otimes \mathcal{K}(l^2))$ such that $v^*v=1-p$ and $vv^*=1$.  

Let $\widehat{l^2}$ denote the Hilbert space $l^2_{ev}\oplus l^2_{od}$, where each summand is a copy of $l^2$, and grade $\widehat{l^2}$ by stipulating that the first summand is even and the second odd.  Define an operator $F$ on 
\begin{equation}\label{decomp}
A\otimes B\otimes \widehat{l^2}\cong (A\otimes B\otimes l^2_{ev})\oplus (A\otimes B\otimes l^2_{od})
\end{equation}
by 
$$
F=\begin{pmatrix} 0 & v^* \\ v & 0 \end{pmatrix}
$$
(where of course the matrix decomposition reflects the direct sum decomposition in line \eqref{decomp}).
The pair $(A\otimes B\otimes \widehat{l^2},F)$ then defines an element $[F]\in KK_0(\C,A\otimes B)$ which corresponds to $[p]\in K_0(A\otimes B)$ under the natural isomorphism $KK_0(\C,A\otimes B)\cong K_0(A\otimes B)$.
\item Let $\mathcal{H}^B$ be the ample $B$-Hilbert space on which $\mathcal{D}(B)$ is defined, and let $\widehat{\mathcal{H}^B}$ denote the Hilbert space $\mathcal{H}^B_{ev}\oplus\mathcal{H}^B_{od}$, which is defined analogously to $\widehat{l^2}$ above.  $\widehat{\mathcal{H}^B}$ is equipped with the natural action of $B$ (by even operators).  Define
$$
G=\begin{pmatrix} 0 & u^* \\ u & 0 \end{pmatrix}\in\mathcal{B}(\widehat{H}^B)
$$
and note that the pair $(\widehat{\mathcal{H}^B},G)$ defines an element $[G]\in KK^0(B,\C)$ that corresponds to $[u]\in K^0(B)$ under the natural isomorphism $KK^0(B,\C)\cong K^0(B)$.
\end{itemize}

Note that to take the Kasparov product of $[F]$ and $[G]$, we must first replace $[G]$ with the element $[\widehat{\mathcal{H}^B}\otimes A,G\otimes 1]\in KK(A\otimes B,A)$; by an abuse of notation, however, we still denote this element $[G]$.

We must now compute the Kasparov product $[F]\otimes [G]\in KK_0(\C,A)$.  The Hilbert-$A$-module for this element is 
\begin{equation}\label{mod}
(A\otimes B\otimes \widehat{l^2})\bigotimes_{A\otimes B} (\widehat{\mathcal{H}^B}\otimes A) \cong A\otimes \widehat{l^2}\otimes\widehat{\mathcal{H}^B}.
\end{equation}
We will use the decomposition 
$$
(A\otimes l^2_{ev}\otimes \mathcal{H}^B_{ev})\oplus (A\otimes l^2_{od}\otimes \mathcal{H}^B_{od})\oplus (A\otimes l^2_{ev}\otimes \mathcal{H}^B_{od})\oplus (A\otimes l^2_{od}\otimes \mathcal{H}^B_{ev})
$$
of this module to write operators on it as $4\times 4$ matrices.  Note now that a $G\otimes 1$-connection (see \cite[Section 18.3]{Blackadar:1998yq}) is given by
$$
\hat{G}=\begin{pmatrix} 0 & 0 & 1\otimes u^* & 0 \\ 0 & 0 & 0 & -1\otimes u^* \\ 1\otimes u & 0 & 0 & 0 \\ 0 & -1\otimes u & 0 & 0 \end{pmatrix},
$$
whence \cite[Proposition 18.10.1]{Blackadar:1998yq} implies that the product $[F]\otimes [G]$ can be represented by the pair
\begin{equation}\label{prod}
(A\otimes \widehat{l^2}\otimes\widehat{\mathcal{H}^B}~,~F\hat{\otimes} 1 +((1-F^2)^\frac{1}{2}\hat{\otimes} 1)\hat{G}).
\end{equation}
Now, the natural (even) action of $A\otimes B\otimes \mathcal{K}(l^2)$ on $A\otimes \mathcal{H}^B\otimes l^2$ extends to the multiplier algebra, so we may treat the operators $p$, $v$ and $v^*$ as acting directly on $A\otimes \mathcal{H}^B\otimes l^2$.  Having adopted this convention, the operator from line \eqref{prod} above is equal to 
\begin{align*} 
& \begin{pmatrix} 0 & v^* \\ v &0 \end{pmatrix}\hat{\otimes 1} +\Big(\begin{pmatrix} 1-v^*v & 0 \\ 0 & 1-vv^*\end{pmatrix}^\frac{1}{2}\hat{\otimes} 1\Big)\begin{pmatrix} 0 & 0 & 1\otimes u^* & 0 \\ 0 & 0 & 0 & -1\otimes u^* \\ 1\otimes u & 0 & 0 & 0 \\ 0 & -1\otimes u & 0 & 0 \end{pmatrix} \\
& =\begin{pmatrix} 0 & 0 & 0 &  v^* \\ 0 & 0 & v & 0 \\ 0 & v^* & 0 & 0 \\ v & 0 & 0 & 0 \end{pmatrix}+\begin{pmatrix} p & 0 & 0 & 0 \\ 0 & 0 & 0 & 0 \\ 0 & 0 & p & 0 \\ 0 & 0 & 0 & 0 \end{pmatrix}\begin{pmatrix} 0 & 0 & 1\otimes u^* & 0 \\ 0 & 0 & 0 & -1\otimes u^* \\ 1\otimes u & 0 & 0 & 0 \\ 0 & -1\otimes u & 0 & 0 \end{pmatrix} \\
& = \begin{pmatrix} 0 & 0 & p(1\otimes u^*) & v^* \\ 0 & 0 & v & 0 \\ p(1\otimes u) & v^* & 0 & 0 \\ v & 0 & 0 & 0 \end{pmatrix}.
\end{align*}
Passing back to the ungraded picture, the class of the cycle
$$
\Big(A\otimes \widehat{l^2}\otimes\widehat{\mathcal{H}^B}~~,~~\begin{pmatrix} 0 & 0 & p(1\otimes u^*) & v^* \\ 0 & 0 & v & 0 \\ p(1\otimes u) & v^* & 0 & 0 \\ v & 0 & 0 & 0 \end{pmatrix}\Big)
$$
in $KK_0(\C,A)$ corresponds under the isomorphism
$$
KK_0(\C,A)\cong K_1\Big(\frac{M(A\otimes\mathcal{K}(l^2\otimes \widehat{\mathcal{H}^B}))}{A\otimes\mathcal{K}(l^2\otimes \widehat{\mathcal{H}^B})}\Big) ~~~(~\cong K_0(A)~)
$$
to the $K_1$-class defined by 
$$
\begin{pmatrix} p(1\otimes u) & v^* \\ v & 0 \end{pmatrix}\in M(A\otimes\mathcal{K}(l^2\otimes \widehat{\mathcal{H}^B}))
$$
(this element is indeed unitary modulo $A\otimes\mathcal{K}(l^2\otimes \widehat{\mathcal{H}^B})$).
Modulo $A\otimes\mathcal{K}(l^2\otimes \widehat{\mathcal{H}^B})$, however, we have that  
$$
\begin{pmatrix} p(1\otimes u) & v^* \\ v & 0 \end{pmatrix}=\begin{pmatrix} p(1\otimes u)p & v^* \\ v & 0 \end{pmatrix}=\begin{pmatrix} p(1\otimes u)p +(1-p) & 0 \\ 0 & 1\end{pmatrix}\begin{pmatrix} p & v^* \\ v & 0\end{pmatrix};
$$
moreover, the second matrix in the product satisfies $X^2=I$, and is thus $K$--theoretically trivial.  Hence the class we have is 
$$
\Big[\begin{pmatrix} p(1\otimes u)p +(1-p) & 0 \\ 0 & 1\end{pmatrix}\Big]\in K_1\Big(\frac{M(A\otimes\mathcal{K}(l^2\otimes \widehat{\mathcal{H}^B}))}{A\otimes\mathcal{K}(l^2\otimes \widehat{\mathcal{H}^B})}\Big);
$$
using the fact that the inclusion 
$$
\frac{A\otimes\mathcal{B}(l^2\otimes \widehat{\mathcal{H}^B})}{A\otimes\mathcal{K}(l^2\otimes \widehat{\mathcal{H}^B})}\hookrightarrow \frac{M(A\otimes\mathcal{K}(l^2\otimes \widehat{\mathcal{H}^B}))}{A\otimes\mathcal{K}(l^2\otimes \widehat{\mathcal{H}^B})}
$$
induces an isomorphism on $K$--theory and the argument of Example \ref{pairex}, however, it is not difficult to see that the image of this in $K_0(A)$ is precisely the same as the slant product $[p]/[u]$ from Definition \ref{slant}.
\end{proof}

\end{document}